\documentclass[12pt]{article}

\usepackage{t1enc}
\usepackage[latin1]{inputenc}
\usepackage[english]{babel}
\usepackage{amsmath,amsthm}
\usepackage{amsfonts}
\usepackage{latexsym}
\usepackage{graphicx}
\usepackage{txfonts}
\usepackage[natural]{xcolor}

\newtheorem{thm}{Theorem}
\newtheorem{lem}[thm]{Lemma}
\newtheorem{cor}[thm]{Corollary}
\newtheorem{conj}[thm]{Conjecture}

\makeatletter

\newcommand{\Rmnum}[1]{\expandafter\@slowromancap\romannumeral #1@}
\makeatother

\begin{document}

\title{Equitable list point arboricity of graphs\thanks{This work is mainly supported by the Specialized
Research Fund for the Doctoral Program of Higher Education (No.\,20130203120021) and is partially supported by the National Natural Science Foundation of China (No.\,11301410, 11201440, 11101243), the Natural Science Basic Research Plan in Shaanxi Province of China (No.\,2013JQ1002), and the Fundamental Research Funds for the Central Universities (No.\,K5051370003, K5051370021).}}
\author{Xin Zhang\thanks{Email address: xzhang@xidian.edu.cn.}\\
{\small Department of Mathematics, Xidian University, Xi'an 710071, China}}
\date{}

\maketitle

\begin{abstract}\baselineskip  0.60cm
A graph $G$ is list point $k$-arborable if, whenever we are given a $k$-list assignment $L(v)$ of colors
for each vertex $v\in V(G)$, we can choose a color $c(v)\in L(v)$ for each vertex $v$ so that each color class induces an acyclic subgraph of $G$, and is equitable list point $k$-arborable if $G$ is list point $k$-arborable and each color appears on at most $\lceil |V(G)|/k\rceil$ vertices of $G$.
In this paper, we conjecture that every graph $G$ is equitable list point $k$-arborable for every $k\geq \lceil(\Delta(G)+1)/2\rceil$ and settle this for complete graphs, 2-degenerate graphs, 3-degenerate claw-free graphs with maximum degree at least 4, and planar graphs with maximum degree at least 8.\\[.2em]

\end{abstract}

\baselineskip  0.60cm

\section{Introduction}

All graphs considered here are simple and undirected. We use $V(G), E(G),\delta(G)$ and $\Delta(G)$ to denote the set of vertices, the set of edges, the minimum degree and the maximum degree of $G$, respectively. For a plane graph $G$, we denote by $F(G)$ the set of faces of $G$. A $k$-, $k^+$- and $k^-$-vertex (resp.\,face) is a vertex (resp.\,face)  of degree $k$, at least $k$ and at most $k$, respectively. A $(k_1,k_2,k_3)$-face is a 3-face that is incident with a $k_1$-vertex, a $k_2$-vertex and a $k_3$-vertex. Other faces such as $(k_1^-,k_2^-,k_3^-)$-face can be defined similarly. For any undefined notions we refer the readers to \cite{Bondy.2008}.

The \emph{point arboricity}, or \emph{vertex arboricity} of $G$, which is introduced by Chartrand \emph{et al.}\,\cite{CK} and denoted by $\rho(G)$, is the minimum number of colors that can be used to color the vertices of G so that each color class induces an acyclic subgraph of $G$. In 2000, Borodin,
Kostochka and Toft \cite{BKT} introduced the list version of point arboricity. A graph $G$ is \emph{list point $k$-arborable} if, whenever we are given a $k$-list assignment $L(v)$ of colors
for each vertex $v\in V(G)$, we can choose a color $c(v)\in L(v)$ for each vertex $v$ so that each color class induces an acyclic subgraph of $G$.
The minimum integer $k$ such that $G$ is list point $k$-arborable, denoted by $\rho_l(G)$, is the \emph{list point arboricity} of $G$. It is proved by Xue and Wu \cite{XW} that $\rho_l(G)\leq \lceil(\Delta(G)+1)/2\rceil$ for every graph $G$.

An equitable list $k$-coloring (not needed to be proper) of $G$ is an assignment of colors to the vertices of $G$ so that the color of each vertex $v\in V(G)$ is chosen from its list $L(v)$ of size $k$ and each color appears on at most $\lceil |V(G)|/k\rceil$ vertices of $G$. This parameter of graphs was introduced by Kostochka, Pelsmajer and West \cite{KPW} and extensively studied by many authors since then.
As we know, the list point arboricity is a chromatic parameter corresponding to a list improper coloring. Therefore, we can naturally
introduce the equitable version of the list point arboricity, which is an equitable list improper chromatic parameter.

A graph $G$ is \emph{equitable list point $k$-arborable} if $G$ is list point $k$-arborable and each color appears on at most $\lceil |V(G)|/k\rceil$ vertices of $G$.
The minimum integer $k$, denoted by $\rho_l^=(G)$, is the \emph{equitable list point arboricity} of $G$.
In this paper, we put forward the following conjectures and confirm them for complete graphs, 2-degenerate graphs, 3-degenerate claw-free graphs with maximum degree at least 4, and planar graphs with maximum degree at least 8.

\begin{conj}\label{conj}
$\rho_l^=(G)\leq \lceil(\Delta(G)+1)/2\rceil$ for every graph $G$.
\end{conj}

\begin{conj}\label{conj2}
Every graph $G$ is equitable list point $k$-arborable for every $k\geq \lceil(\Delta(G)+1)/2\rceil$.
\end{conj}

\section{Main Results and their proofs}

\begin{thm}\label{Kn}
The complete graph $K_n$ is equitable list point $k$-arborable for every $k\geq \lceil\frac{n}{2}\rceil$, and moreover, $\rho_l^=(K_n)=\lceil\frac{n}{2}\rceil$. \end{thm}

\begin{proof}
Let $v_1,v_2,\ldots,v_n$ be the vertices of $K_n$ and let $L(v)$ be a $k$-list assignment of colors
for each vertex $v\in V(K_n)$. If $k\geq n$, then it is easy to check that $K_n$ has a list $k$-coloring $c$ with $c(v_i)\neq c(v_j)$ for each $i\neq j$, which implies that $K_n$ is equitable list point $k$-arborable.
We now assume that $k<n$ and construct as follows a list $k$-coloring of $K_n$ by coloring $v_1,v_2,\ldots,v_n$  in sequence.
First, color $v_1$ with $c(v_1)\in L(v_1)$ and color $v_i$ with $c(v_i)\in L(v_i)\setminus \{c(v_1),\ldots,c(v_{i-1})\}$ for each $2\leq i\leq k$. We then color $v_j$ with a color from $L(v_j)$ that is already used at most once on the vertices with lower subscript for each $k+1\leq j\leq n$.
Since $k\geq \lceil\frac{n}{2}\rceil$, all of the above steps can be done. Moreover, one can see that each color under $c$ appears on at most two vertices of $K_n$, which implies that $K_n$ is list point $k$-arborable. Since $\lceil\frac{n}{k}\rceil=2$, $K_n$ is equitable list point $k$-arborable and $\rho_l^=(K_n)\leq \lceil\frac{n}{2}\rceil$. On the other hand, we have $\rho_l^=(K_n)\geq \rho(K_n)=\lceil\frac{n}{2}\rceil$, which implies that $\rho_l^=(K_n)=\lceil\frac{n}{2}\rceil$.
\end{proof}

By Theorem \ref{Kn}, Conjectures \ref{conj} and \ref{conj2} hold for complete graphs, and moreover, the upper bound in Conjecture \ref{conj} and the lower bound in Conjecture \ref{conj2} are sharp if they are right.

\begin{lem}\label{label}
Let $S=\{x_1,\cdots,x_k\}$, where $x_1,\cdots,x_k$ are distinct vertices in $G$. If $G-S$ is equitable list point $k$-arborable and $|N(x_i)\setminus S|\leq 2i-1$ for every $1\leq i\leq k$, then $G$ is equitable list point $k$-arborable.
\end{lem}

\begin{proof}
Suppose that $G-S$ has an equitable list $k$-coloring $c$ such that each color set of $c$ induces an acyclic subgraph. Note that every color set of $c$ is of size at most $\lceil\frac{|V(G)|-k}{k}\rceil$. Assign $x_k$ a color in its list that used at most once on the neighborhood of $x_k$. Since $N(x_k)\leq 2k-1$ and $|L(x_k)|=k$, this can be done. We then assign $x_{k-1},\ldots,x_2$ and $x_1$ in sequence a color in its list that is different from the one assigned to the vertices with higher subscript and used at most once on its neighborhood. All of those steps can be completed since $|N(x_i)\setminus S|\leq 2i-1$ for every $1\leq i\leq k$. Therefore, the coloring $c$ can be extended to a list $k$-coloring of $G$ so that each color set induces an acyclic subgraph of order of at most
$\lceil\frac{|V(G)|}{k}\rceil$. Hence $G$ is equitable list point $k$-arborable.
\end{proof}

A graph $G$ is \emph{$k$-degenerate} if every subgraph of $G$ has a vertex of degree at most $k$. Using Lemma \ref{label}, we can confirm Conjectures \ref{conj} and \ref{conj2} for 2-degenerate graphs, a wide class including series-parallel graphs, outerplanar graphs, planar graphs with girth at least 6, etc.

\begin{thm}\label{thm1}
Every $2$-degenerate graph $G$ is equitable list point $k$-arborable for every $k\geq \lceil\frac{\Delta(G)+1}{2}\rceil$.
\end{thm}

\begin{proof}
If $\Delta(G)\leq 1$, then this result is trivial, so we assume that $\Delta(G)\geq 2$ and $k\geq 2$.
Let $uv$ be an edge of $G$ with $d(u)\leq 2$.
We now construct the set $S=\{x_1,\cdots,x_k\}$ by labeling $u$ and $v$ with $x_1$ and $x_k$, and filling the remaining unspecified positions in $S$ from highest to lowest indices with a
vertex of degree at most 2 in the graph obtained from $G$ by deleting the vertices already chosen for $S$. Those $2^-$-vertices always exist since $G$ is 2-degenerate. Moreover, we have $|N(x_1)\setminus S|\leq 1$, $|N(x_k)\setminus S|\leq \Delta(G)-1\leq 2k-1$ and $|N(x_i)\setminus S|\leq 2\leq 2i-1$ for every $2\leq i\leq k-1$. Therefore, by Lemma \ref{label}, $G$ is equitable list point $k$-arborable if $G-S$ is equitable list point $k$-arborable. Hence we can complete the proof by induction on the order of $G$.
\end{proof}

A graph $G$ is \emph{claw-free} if any 3-vertex subgraph of $G$ can not induce a $K_{1,3}$. Note that any graph obtained from a claw-free graph by removing some vertices is also claw-free.

\begin{thm}\label{thm2}
Every $3$-degenerate claw-free graph $G$ is equitable list point $k$-arborable for every $k\geq \max\{\lceil\frac{\Delta(G)+1}{2}\rceil,3\}$.
\end{thm}

\begin{proof}
Let $ux$ be an edge of $G$ with $d(u)\leq 3$. If $d(u)\leq 2$, then we can prove this result by the same arguments as in the proof of Theorem \ref{thm1}. Hence we assume that $d(u)=3$. Let $y$ and $z$ be another two neighbors of $u$. Since $G$ is claw-free, we can assume, without loss of generality, that $yz\in E(G)$.
We now construct the set $S=\{x_1,\cdots,x_k\}$ by labeling $u,y$ and $z$ with $x_1,x_{k-1}$ and $x_k$, and filling the remaining unspecified positions in $S$ from highest to lowest indices with a vertex of degree at most 3 in the graph obtained from $G$ by deleting the vertices already chosen for $S$.
Since $|N(x_1)\setminus S|\leq 1$, $|N(x_{k-1})\setminus S|\leq \Delta(G)-2\leq 2(k-1)-1$, $|N(x_k)\setminus S|\leq \Delta(G)-2\leq 2k-1$ and $|N(x_i)\setminus S|\leq 3\leq 2i-1$ for every $2\leq i\leq k-2$, $G$ is equitable list point $k$-arborable by Lemma \ref{label} if $G-S$ is equitable list point $k$-arborable.
This makes us possible to complete the proof by induction on the order of $G$.
\end{proof}

\begin{cor}\label{cor1}
Every $3$-degenerate claw-free graph $G$ with maximum degree $\Delta\geq 4$ is equitable list point $k$-arborable for every $k\geq \lceil\frac{\Delta+1}{2}\rceil$.
\end{cor}

\begin{thm}\label{thm3}
Every planar graph is equitable list point $k$-arborable for every $k\geq \max\{\lceil\frac{\Delta(G)+1}{2}\rceil,5\}$.
\end{thm}

\begin{proof}
Let $G$ be a planar graph such that $G$ is not equitable list point $k$-arborable and every subgraph $G'\subset G$ is equitable list point $k$-arborable.

\vspace{2mm}\noindent Claim 1. \emph{$\delta(G)\geq 2$.}

\proof Suppose, to the contrary, that $G$ has a vertex $v$ of degree at most 1. By the minimality of $G$, $G-v$ is equitable list point $k$-arborable and each color appears on at most $\lceil\frac{|V(G)|-1}{k}\rceil$ vertices of $G-v$. Let $S$ be the set of colors that appears on exactly $\lceil\frac{|V(G)|-1}{k}\rceil$ vertices of $G-v$ and let $s=|S|$. Since $s\lceil\frac{|V(G)|-1}{k}\rceil\leq |V(G)|-1$, $s\leq k$ and the upper bound can be attained only if  $|V(G)|\equiv 1$ (mod $k$). If $s<k$, then give $v$ a color $c(v)$ from $L(v)\setminus S$ which has size at least $k-(k-1)=1$. Since the color $c(v)$ appears on at most $\lceil\frac{|V(G)|-1}{k}\rceil-1$ vertices of $G-v$, $c(v)$ appears on at most $\lceil\frac{|V(G)|-1}{k}\rceil\leq \lceil\frac{|V(G)|}{k}\rceil$ vertices of $G$, which implies that $G$ is equitable list point $k$-arborable. If $s=k$, then $G-v$ is incident with exactly $k$ colors and $\lceil\frac{|V(G)|-1}{k}\rceil+1=\lceil\frac{|V(G)|}{k}\rceil$. Hence we can color $v$ with an arbitrary color form its list to ensure that $G$ is equitable list point $k$-arborable.\hfill$\Box$

\vspace{2mm}\noindent Claim 2. \emph{If $\delta(G)=2$, then $G$ contains only one $3^-$-vertex.}

\proof Suppose that $G$ has a 2-vertex $u$ and a $3^-$-vertex $v$ with $u\neq v$. Let $x$ and $y$ be the neighbors of $u$. Note that $v$ may be $x$ or $y$. Label $u$ and $v$ with $x_1$ and $x_2$, and label the unique vertex in $N(u)\setminus \{v\}$ with $x_k$.
Fill the remaining unspecified positions in $S=\{x_1,\cdots,x_k\}$ from highest to lowest indices with a
vertex of degree at most 5 in the graph obtained from $G$ by deleting the vertices already chosen for $S$. Since every planar graph is 5-degenerate, this can be done. One can check that $|N(x_i)\setminus S|\leq 2i-1$ for every $1\leq i\leq k$. Therefore, $G$ is equitable list point $k$-arborable by Lemma \ref{label}, since $G-S$ is so by the minimality of $G$. \hfill$\Box$

\vspace{2mm}\noindent Claim 3. \emph{Every $3$-vertex is adjacent only to $5^+$-vertices.}

\proof Suppose, to the contrary, that a 3-vertex $u$ is adjacent to a $4^-$-vertex $v$. Label $u,v$ and one vertex in $N(u)\setminus \{v\}$ with $x_1,x_2$ and $x_k$, and fill the remaining unspecified positions in $S$ as in Lemma \ref{label} by the similar way as in the proof of Claim 2. Since $G-S$ is equitable list point $k$-arborable, $G$ is also equitable list point $k$-arborable by Lemma \ref{label}.\hfill$\Box$

\vspace{2mm}\noindent Claim 4. \emph{If $\delta(G)=3$, then $\Delta(G)\geq 6$.}

\proof Suppose, to the contrary, that $\Delta(G)\leq 5$. Let $u$ be a 3-vertex and let $v,w$ be two neighbors of $u$. By Claim 3, $d(v)=d(w)=5$.
Let $x$ be a neighbor of $v$ that is different from $u$ and $w$. We now label $u,v,w$ and $x$ by $x_1,x_2,x_3$ and $x_k$, and fill the remaining unspecified positions in $S$ as in Lemma \ref{label} by the similar way as in the proof of Claim 2. Hence by Lemma \ref{label}, $G$ is equitable list point $k$-arborable since $G-S$ is so by the minimality of $G$. \hfill$\Box$

\vspace{2mm}\noindent Claim 5. \emph{If $\delta(G)=3$ and there are at least two $3$-vertices, then every $3$-vertex is adjacent only to maximum degree vertices.}

\proof Let $u$ and $v$ be different two 3-vertices. If $uv\in E(G)$, then label $u,v$ and one vertex in $N(u)\setminus \{v\}$ by $x_1,x_2$ and $x_k$, respectively, and fill the remaining unspecified positions in $S$ as in Lemma \ref{label} by the similar way as in the proof of Claim 2. This operation implies that $G$ is equitable list point $k$-arborable by Lemma \ref{label}, since $G-S$ is so by the minimality of $G$, a contradiction. Therefore, we assume that $uv\not\in E(G)$. Suppose, to the contrary, that $u$ is adjacent to a $(\Delta(G)-1)^-$-vertex $w$. In this case we label $u,v,w$ and one vertex in $N(u)\setminus \{w\}$ with $x_1,x_2,x_{k-1}$ and $x_k$, respectively, and fill the remaining unspecified positions in $S$ as in Lemma \ref{label} properly. Since $G-S$ is equitable list point $k$-arborable, $G$ is also equitable list point $k$-arborable by Lemma \ref{label}.\hfill$\Box$

\vspace{2mm}\noindent Claim 6. \emph{If $\delta(G)\leq 3$, then there are at most three $3$-faces that is incident with a $3^-$-vertex.}

\proof If there are four 3-faces that is incident with a $3^-$-vertex, then there must be two 3-faces $f_1=uvw$ and $f_2=xyz$ with $d(u)\leq 3$, $d(x)\leq 3$ and $u\neq x$. Label $u,x,v$ and $w$ with $x_1,x_2,x_{k-1}$ and $x_k$, respectively, and fill the remaining unspecified positions in $S$ as in Lemma \ref{label} by the similar way as in the proof of Claim 2. Since $G-S$ is equitable list point $k$-arborable, $G$ is also equitable list point $k$-arborable by Lemma \ref{label}.\hfill$\Box$

\vspace{2mm}\noindent Claim 7. \emph{Let $f=uvw$ be a $3$-face.
If $d(u)\leq 4$, then $d(v)\geq 6$ and $d(w)\geq 6$.}

\proof Suppose, to the contrary, that $d(v)\leq 5$. If $d(u)=4$, then label $u,v,w$ and a vertex in $N(u)\setminus \{v,w\}$ by $x_1,x_2,x_{k-1}$ and $x_k$, respectively. If $d(u)\leq 3$, then label $u,v$ and $w$ by $x_1,x_2$ and $x_k$, respectively, In any case we can construct the set $S$ as in Lemma \ref{label} by filling the remaining unspecified positions from highest to lowest indices with a
vertex of degree at most 5 in the graph obtained from $G$ by deleting the vertices already chosen for $S$. By Lemma \ref{label}, $G$ is equitable list point $k$-arborable since $G-S$ is so by the minimality of $G$, a contradiction. Therefore, we have $d(v)\geq 6$, and by symmetry, $d(w)\geq 6$. \hfill$\Box$

\vspace{2mm}\noindent Claim 8. \emph{Let $f=uvw$ be a $3$-face that is incident only with $4^+$-vertices.
If $\delta(G)\leq 3$ and $d(u)=4$, then $d(v)\geq 8$ and $d(w)\geq 8$.}

\proof Let $x$ be the vertex with the minimum degree. Suppose, to the contrary, that $d(v)\leq 7$. Label $u,x,v,w$ and one vertex in $N(u)\setminus\{v,w\}$ with $x_1,x_2,x_3,x_{k-1}$ and $x_k$, and fill the remaining unspecified positions in $S$ as in Lemma \ref{label} by the similar way as in the proof of Claim 2. By Lemma \ref{label}, $G$ is equitable list point $k$-arborable since $G-S$ is so by the minimality of $G$, a contradiction implying that $d(v)\geq 8$. By symmetry we also have $d(w)\geq 8$.\hfill$\Box$

\vspace{2mm}\noindent Claim 9. \emph{Let $f_1=uvw$ be a $3$-face and let $f_2$ be the face sharing the common edge $uv$ with $f_1$.
If $d(u)=5$, $5\leq d(v)\leq 6$ and $5\leq d(w)\leq 7$, then $d(f_2)\geq 4$.}

\proof Suppose, to the contrary, that $f_2=uvx$ is a 3-face. Label $u,v,w,x$ and one vertex in $N(u)\setminus\{x,v,w\}$ with $x_1,x_2,x_3,x_{k-1}$ and $x_k$, and fill the remaining unspecified positions in $S$ as in Lemma \ref{label} by the similar way as in the proof of Claim 2. Since $G-S$ is equitable list point $k$-arborable, $G$ is also equitable list point $k$-arborable by Lemma \ref{label}.\hfill$\Box$

\vspace{2mm}\noindent Claim 10. \emph{Let $f_1=uvw$ be a $3$-face and let $f_2$ be the face sharing the common edge $uv$ with $f_1$.
If $d(u)=6$ and $d(v)=4$, then $d(f_2)\geq 4$.}

\proof Suppose, to the contrary, that $f_2=uvx$ is a 3-face. Label $v,u,w$ and $x$ by $x_1,x_2,x_{k-1}$ and $x_k$,
and fill the remaining unspecified positions in $S$ as in Lemma \ref{label} by the similar way as in the proof of Claim 2.
By Lemma \ref{label}, $G$ is equitable list point $k$-arborable since $G-S$ is so by the minimality of $G$.\hfill$\Box$

\vspace{2mm} We now prove Theorem \ref{thm3} by discharging. First, assign to each element $x\in V(G)\cup F(G)$ an initial charge $c(x)=d(x)-4$. By Euler's formula, it is easy to see that $\sum_{x\in V(G)\cup F(G)}c(x)=-8$. In the next, we will reassign a new charge, denoted by $c'(x)$, to each
$x\in V(G)\cup F(G)$ according to the following discharging rules. Since our rules only move charge around, and do not affect the sum, we have $\sum_{x\in
V(G)\cup F(G)}c'(x)=\sum_{x\in V(G)\cup F(G)}c(x)=-8$.

\vspace{2mm} R1. If $v$ is a 3-vertex that is adjacent only to maximum degree vertices, then $v$ receives $\frac{1}{3}$ from each of its neighbors.

R2. Let $f=uvw$ be 3-face that is incident only with $4^+$-vertices.

\indent\indent R2.1. If $\delta(G)\leq 3$ and $d(u)=6$, then $u$ sends $\frac{1}{3}$ to $f$.

\indent\indent R2.2. If $\delta(G)\leq 3$ and $d(u)=7$, then $u$ sends $\frac{3}{7}$ to $f$.

\indent\indent R2.3. If $d(u)\geq 8$, then $u$ sends $\frac{1}{2}$ to $f$.


\indent\indent R2.4. If $\delta(G)=4$ and $d(u)=4$, then each of $v$ and $w$ sends $\frac{1}{2}$ to $f$.

R3. If $f=uvw$ is a 3-face that is incident only with $5^+$-vertices.

\indent\indent R3.1. If $\delta(G)\geq 4$ and $d(u)=6$, then $u$ sends $\frac{1}{3}$ to $f$.

\indent\indent R3.2. If $\delta(G)\geq 4$ and $d(u)=7$, then $u$ sends $\frac{3}{7}$ to $f$.


R4. If there is a face $f$ that is incident only with $4^+$-vertices and has negative charge $-\gamma$ after applying R2 and R3, then $f$ receives $\frac{\gamma}{n_5}$ from each of its incident $5$-vertices, where $n_5$ denotes the number of 5-vertices that are incident with $f$.

\vspace{2mm} Since 2-vertices, 4-vertices and $4^+$-faces are not involved in the above rules by Claim 3, we have $c'(v)=c(v)=-2$ for a 2-vertex $v$, $c'(v)=c(v)=0$ for a 4-vertex $v$ and $c'(f)=c(f)\geq 0$ for a $4^+$-face.
If $G$ has at least two 3-vertices, then $\delta(G)=3$ by Claim 2 and
$c'(v)=-1+3\times\frac{1}{3}=0$ for a 3-vertex $v$ by Claim 5 and R1. If $G$ has exactly one 3-vertex $v$, then $c'(v)\geq c(v)=-1$. If $f$ is a 3-face that is incident with a $3^-$-vertex, then $c'(f)=c(f)=-1$.
Let $f$ be a 3-face that is incident only with $4^+$-vertices.
If $f$ is incident only with $6^+$-vertices, then $c'(f)\geq -1+3\times\frac{1}{3}=0$ by R2 and R3. If $f$ is incident with at least one 5-vertex, then R4 guarantees that $c'(f)\geq 0$. If $f$ is incident with a 4-vertex and $\delta(G)\leq 3$,
then $f$ is incident with at least two $8^+$-vertices by Claim 8, which implies that $c'(f)\geq -1+2\times\frac{1}{2}=0$ by R2.3.
If $f$ is incident with a 4-vertex and $\delta(G)=4$, then $c'(f)= -1+2\times\frac{1}{2}=0$ by R2.4.

We now estimate the final charges of $5^+$-vertices. By R1--R4, $5^+$-vertices only send charges to its adjacent 3-vertices and incident 3-faces that are incident only with $4^+$-vertices. From now on, we call the 3-face that is incident only with $4^+$-vertices \emph{considerable 3-faces}.

Case 1. $\delta(G)=2$ or $\delta(G)=5$.

 By Claim 2, $G$ has no 3-vertices.
Let $v$ be a 5-vertex. Since every 3-face that is incident with $v$ is incident with only $5^+$-vertices by Claim 7, $v$ sends such a 3-face at most $\frac{1}{3}$ by R2 and R4.
If $v$ is incident with at least two 4-faces, then it is easy to see that $c'(v)\geq 1-3\times\frac{1}{3}=0$. Thus we assume that $v$ is incident with at most one 4-face, which implies, by Claim 9, that $v$ is incident with at most two considerable $(5,6^-,7^-)$-faces.
If $v$ is incident with two considerable $(5,6^-,7^-)$-faces, then $v$ is incident with a $4^+$-face and two $(5,7,7^+)$-faces, which implies that
$c'(v)\geq 1-2\times\frac{1}{3}-2\times (1-2\times\frac{3}{7})=\frac{1}{21}>0$ by R2.2, R2.3, R3.2 and R4.
If $v$ is incident with exactly one considerable $(5,6^-,7^-)$-face, then by Claim 9, $v$ is incident with a $4^+$-face and three $(5,7^+,7^+)$-faces, or a $4^+$-face, a $(5,7,7^+)$-face, a $(5,7^+,8^+)$-face and a $(5,5^+,8^+)$-face, or a $4^+$-face, a $(5,7,8^+)$-face and two $(5,5^+,8^+)$-faces, which implies that $c'(v)\geq 1-\frac{1}{3}-\max\{3\times(1-2\times\frac{3}{7}), (1-2\times\frac{3}{7})+(1-\frac{3}{7}-\frac{1}{2})+\frac{1}{2}\times(1-\frac{1}{2}), (1-\frac{3}{7}-\frac{1}{2})+2\times\frac{1}{2}\times(1-\frac{1}{2})\}=\frac{2}{21}>0$ by R2.2, R2.3, R3.2 and R4.
We now assume that every considerable 3-face that is incident with $v$ is either $(5,7^+,7^+)$-face or $(5,6^-,8^+)$-face. If $v$ is incident with at most four 3-faces or is incident with a $(5,8^+,8^+)$-face, then $c'(v)\geq 1-4\times\max\{1-2\times\frac{3}{7},\frac{1}{2}\times(1-\frac{1}{2})\}=0$ by R2.2, R2.3, R3.2 and R4. If $v$ is incident with at most two considerable $(5,6^-,8^+)$-faces, then $c'(v)\geq 1-2\times \frac{1}{2}\times(1-\frac{1}{2})-3\times (1-2\times\frac{3}{7})=\frac{1}{14}>0$ by R2.2, R2.3, R3.2 and R4. Therefore, we shall only consider the case when $v$ is incident with five 3-faces and at least three of them are considerable $(5,6^-,8^+)$-faces. However, one can easy to check that if this case occurs then $v$ is incident with a $(5,8^+,8^+)$-face and we come back to the case considered before.
If $v$ is a 6-vertex, then $c'(v)\geq 2-6\times\frac{1}{3}=0$ by R2.1 and R3.1. If $v$ is a 7-vertex, then $c'(v)\geq 3-7\times\frac{3}{7}=0$ by R2.2 and R3.2. If $v$ is a $8^+$-vertex, then $c'(v)\geq d(v)-8-\frac{1}{2}d(v)\geq 0$ by R2.3.

Case 2. $\delta(G)=3$.

By Claim 4, we have $\Delta(G)\geq 6$. Under this condition, one can prove that $c'(v)\geq 0$ for every 5-vertex $v$ by the same arguments as in Case 1, since $5$-vertices would not send charges to 3-vertices by R1.
Let $v$ be a $6^+$-vertex and let $n_3$ be the number of 3-vertices that are adjacent to $v$. It is easy to see that $v$ is incident with at most $d(v)-n_3-1$ considerable 3-faces. Therefore, if $d(v)=6$, then $c'(v)\geq d(v)-4-\frac{1}{3}n_3-\frac{1}{3}(d(v)-n_3-1)= \frac{1}{3}>0$ by R1 and R2.1, if $d(v)=7$, then
$c'(v)\geq d(v)-4-\frac{1}{3}n_3-\frac{3}{7}(d(v)-n_3-1)\geq \frac{3}{7}>0$ by R1 and R2.2, and if $d(v)\geq 8$, then
$c'(v)\geq d(v)-4-\frac{1}{3}n_3-\frac{1}{2}(d(v)-n_3-1)\geq \frac{1}{2}>0$ by R1 and R2.3.

Case 3. $\delta(G)=4$.

By Claim 7, R2.4 will not apply to any 5-vertex, thus by the same arguments as in Case 1 one can prove that $c'(v)\geq 0$ for every 5-vertex $v$. If $v$ is a $8^+$-vertex, then by R2.3 and R2.4, $c'(v)\geq d(v)-4-\frac{1}{2}d(v)\geq 0$.


Let $v$ be a 6-vertex. If $v$ is incident with no $(4,6,6^+)$-faces, then by R3.1, $c'(v)\geq 2-6\times\frac{1}{3}=0$.
If $v$ is incident with exactly one $(4,6,6^+)$-face, then by Claim 10, $v$ is incident with a $4^+$-face, which implies that $c'(v)\geq 2-\frac{1}{2}-4\times\frac{1}{3}=\frac{1}{6}>0$ by R2.4 and R3.1. If $v$ is incident with two $(4,6,6^+)$-faces, then by Claim 10, $v$ is incident with at least one $4^+$-face, which implies that $c'(v)\geq 2-2\times\frac{1}{2}-3\times\frac{1}{3}=0$ by R2.4 and R3.1. If $v$ is incident with at least three $(4,6,6^+)$-faces, then $v$ is incident with at least two $4^+$-faces by Claim 10, which implies that $c'(v)\geq 2-4\times\frac{1}{2}=0$ by R2.4.

Let $v$ be a 7-vertex. If $v$ is incident with no $(4,7,6^+)$-faces, then by R3.2, $c'(v)\geq 3-7\times\frac{3}{7}=0$.
If $v$ is incident with exactly one $(4,7,6^+)$-face, then by Claim 10, $v$ is incident with a $4^+$-face, which implies that $c'(v)\geq 3-\frac{1}{2}-5\times\frac{3}{7}=\frac{5}{14}>0$ by R2.4 and R3.2. If $v$ is incident with two $(4,7,6^+)$-faces, then by Claim 10, $v$ is incident with at least one $4^+$-face, which implies that $c'(v)\geq 3-2\times\frac{1}{2}-4\times\frac{3}{7}=\frac{2}{7}>0$ by R2.4 and R3.2. If $v$ is incident with at least three $(4,7,6^+)$-faces, then $v$ is incident with at least two $4^+$-faces by Claim 10, which implies that $c'(v)\geq 3-5\times\frac{1}{2}=\frac{1}{2}>0$ by R2.4.

\vspace{2mm} Until now, we have proved that the final charges of $4^+$-vertices, $4^+$-faces and considerable 3-faces are nonnegative. Therefore, if $\delta(G)=2$, then $\sum_{x\in V(G)\cup F(G)}c'(x)\geq -2-1=-3$, since there are no 3-vertices by Claim 2 and the unique 2-vertex has final charge -2 and there may be a 3-face that is incident with this 2-vertex with final charge -1. If $\delta(G)=3$ and $G$ has only one 3-vertex, then this 3-vertex has final charge at least -1 and there may be at most three 3-faces that are incident with it, any of which has final charge -1. This implies that $\sum_{x\in V(G)\cup F(G)}c'(x)\geq -1-3=-4$.
If $\delta(G)=3$ and $G$ has at least two 3-vertices, then by the above arguments we know that every 3-vertex has final charge 0, and by Claim 6 there are at most three 3-faces that is incident with a 3-vertex, any of which has final charge -1. This implies that $\sum_{x\in V(G)\cup F(G)}c'(x)\geq -3$. If $\delta(G)\geq 4$, then it is easy to see that $\sum_{x\in V(G)\cup F(G)}c'(x)\geq 0$. All in all, we obtain that $\sum_{x\in V(G)\cup F(G)}c'(x)\geq -4$, contradicting the fact that $\sum_{x\in V(G)\cup F(G)}c'(x)=-8$.
\end{proof}

\begin{cor}\label{cor2}
Every planar graph with maximum degree $\Delta\geq 8$ is equitable list point $k$-arborable for every $k\geq \lceil\frac{\Delta+1}{2}\rceil$.
\end{cor}

\end{document}